\documentclass{amsart}

\newtheorem{theorem}{Theorem}[section]
\newtheorem{lemma}[theorem]{Lemma}
\newtheorem{proposition}[theorem]{Proposition}
\newtheorem{corollary}[theorem]{Corollary}
\newtheorem{claim}[theorem]{Claim}

\theoremstyle{definition}
\newtheorem{definition}[theorem]{Definition}
\newtheorem{example}[theorem]{Example}

\theoremstyle{remark}
\newtheorem{remark}[theorem]{Remark}

\numberwithin{equation}{section}

\begin{document}
\title[Generalized torsion for knot groups]{Generalized torsion, unique root property and Baumslag--Solitar relation for knot groups}

\author{Keisuke Himeno}
\address{Graduate School of Advanced Science and Engineering , Hiroshima University,
1-3-1 Kagamiyama, Higashi-hiroshima, 7398526, Japan}
\email{m216754@hiroshima-u.ac.jp}

\author{Kimihiko Motegi}
\address{Department of Mathematics, Nihon University, 
3-25-40 Sakurajosui, Setagaya-ku, 
Tokyo 1568550, Japan}
\email{motegi.kimihiko@nihon-u.ac.jp}
\thanks{The second named author has been partially supported by JSPS KAKENHI Grant Number 19K03502 and Joint Research Grant of Institute of Natural Sciences at Nihon University for 2021. }

\author{Masakazu Teragaito}
\address{Department of Mathematics Education, Hiroshima University,
Higashi-hiroshima 7398524, Japan}
\email{teragai@hiroshima-u.ac.jp}
\thanks{The third named author has been partially supported by JSPS KAKENHI Grant Number JP20K03587.}





\begin{abstract}
Let $G$ be a group. 
If an equation $x^n = y^n$ in $G$ implies $x = y$ for any elements $x$ and $y$, 
then $G$ is called an $R$--group. 
It is completely understood which knot groups are $R$--groups. 
Fay and Walls introduced $\bar{R}$--group in which the normalizer and the centralizer of an isolator of $\langle x \rangle$ coincide for any non-trivial element $x$. 
It is known that $\bar{R}$--groups and $R$--groups share many interesting properties and $\bar{R}$--groups are necessarily $R$--groups. 
However, in general, the converse does not hold. 
We will prove that these classes are the same for knot groups. 
In the course of the proof, we will determine knot groups with generalized torsion of order two.
\end{abstract}

\maketitle

\section{Introduction}\label{sec:intro}

Let $G$ be a group. 
Then the most elementary equation $x^n = y^n$ leads us to the notion of $R$--groups (\cite{B,K}).

\begin{definition}[$R$--group]
A group $G$ is called an \textit{$R$--group} if it has the \textit{unique root property\/}:
an equation $x^n=y^n$ for some non-zero integer $n$ in $G$ implies $x=y$.
\end{definition}

$R$--groups form an important class of torsion-free groups, and some relation to abstract commensurators is studied in \cite{BB}.
It is known that any torsion-free word-hyperbolic group is an $R$--group (for example, see \cite[Lemma 2.2]{BB}).

It is easy to observe that the knot group of a torus knot is not an $R$--group.
Let $G(K)$ be the knot group of a torus knot $K=T(p,q)$.
Then $G(K)$ has a presentation $\langle a, b \mid a^p=b^q\rangle$.
Thus the equation $x^p=a^p$ has solutions $x=a$, $(ba)^{-1}a(ba)$, $(ba)^{-2}a(ba)^2$, and so on.

In 1974, Murasugi \cite{M} gave a sufficient condition for the knot group of a fibered knot
to be an $R$--group.
In particular, he showed that the knot group of the figure-eight knot is an $R$--group.
We should remark that Murasugi's work locates before the works of Jaco--Shalen \cite{JS}, Johannson \cite{J}.

A complete characterization of knot groups which are $R$--groups
is contained in \cite[Proposition 32.4]{J}.
(Johannson discusses the fundamental groups of Haken manifolds, more generally.)
It also follows from \cite{JS}, although it is not explicitly stated.

\begin{theorem}[\cite{JS,J}]\label{thm:main}
Let $K$ be a knot in the $3$--sphere $S^3$, and let $E(K)$ be the exterior and $G(K)=\pi_1(E(K))$.
Then $G(K)$ is an $R$--group if and only if
$E(K)$ contains 
neither a torus knot space nor a cable space as a decomposing piece 
of the torus decomposition.
In particular, the knot group of any hyperbolic knot is an $R$--group.
\end{theorem}

\medskip

In 1999,
Fay and Walls \cite{FW} introduced $\bar{R}$--groups, 
which share many interesting properties with $R$--groups. 

\begin{definition}[$\bar{R}$--group]
A group $G$ is called an \textit{$\bar{R}$--group\/}
if $G$ is torsion-free and the normalizer and the centralizer of the isolator subset
\[
I\langle x\rangle = \{g\in G\mid g^n\in \langle x \rangle\ \text{for some positive integer $n$}\}
\]
 of the cyclic group $\langle x\rangle$ coincide for any $x\in G$.
\end{definition}

It is known that any $\bar{R}$--group is an $R$--group \cite{FW}; see also Lemma~\ref{Rbar_R}. 
However, in general, 
there exist $R$--groups which are not $\bar{R}$--groups (\cite{FW}).
Our main result in this paper asserts that this is not the case among knot groups.

\begin{theorem}
\label{thm:RRbar}
For knot groups, two classes of $R$--groups and $\bar{R}$--groups coincide.
\end{theorem}

In the proof of Theorem~\ref{thm:RRbar}, 
we use a characterization of $\bar{R}$--groups using Baumslag-Solitar relation. 
See Section \ref{sec:Rbar} for the characterization.
This characterization enables us to relate a knot group $G(K)$ being an $\bar{R}$--group and 
being $R$--group using generalized torsion elements defined below. 

In $G$, a non-trivial element $g$ is called a \textit{generalized torsion element\/}
if some non-empty finite product of its conjugates yields the identity.
That is, the equation
\begin{equation}\label{eq:g-torsion}
g^{a_1}g^{a_2}\dots g^{a_n}=1
\end{equation}
holds for some $a_1,a_2,\dots a_n\in G$ and $n\ge 2$, where $g^a=a^{-1}ga$.
The minimum number of conjugates yielding the identity is called the \textit{order\/} of $g$ (\cite{IMT0}).
Since a generalized torsion element is not the identity, its order is at least two.
A typical example is the fundamental group of the Klein bottle.
It has a presentation $\langle a,b \mid a^{-1}ba=b^{-1}\rangle$.
The relation shows $b^a b=1$, so the generator $b$ is a generalized torsion element of order two.

As a generalization of torsion-free groups, 
Fuchs \cite{Fuchs} introduced $R^*$--groups. 

\begin{definition}[$R^*$--group]
A group $G$ is called an \textit{$R^*$--group\/}
if it has no generalized torsion.
\end{definition}

\medskip

Now it should be worth noting some relation among ordering of group, generalized torsion and $R$--groups. 
Recall that a \textit{bi-ordering\/} in a group $G$ is a strict total ordering $<$
which is invariant under left and right multiplications, that is,
\[
x<y \Longrightarrow gxh<gyh  \quad  \text{for $g,h\in G$}.
\]
If $G$ admits a bi-ordering, then it is said to be \textit{bi-orderable}.
Bi-orderable groups are $R$--groups (see \cite{CR}), but
the converse is not true (\cite[p.127]{MR}). 

If we require only the invariance under left multiplication, then
$G$ is said to be \textit{left-orderable}.
It is well known that all knot groups are left-orderable \cite{BRW,HS}.
In \cite{N}, Neuwirth asked if a knot group can be bi-orderable. 
Perron and Rolfsen \cite{PR} gave a sufficient condition for the knot
group of a fibered knot to be bi-orderable, and showed that
the group of the figure-eight knot is bi-orderable.
Since then, there are various results on bi-orderable knot groups (\cite{CGW,CDN, CR,CR2,I,I2,KR,Y}),
but there seems to be no characterization of them.

It is well known that bi-orderable groups are $R^*$--groups, but
there are $R^*$--groups which are not bi-orderable (\cite{Bl,BL,MR}). 
Also, it is not known whether any $R^*$--group is left-orderable or not.

For $3$--manifold groups, including knot groups, we conjecture that two classes of  $R^*$--groups and bi-orderable groups
coincide \cite{MT}.
Although this conjecture is verified for various knot groups (\cite{HT,MT2}), it still remains to be open, in general.

For the relation between $R^*$--groups and $\bar{R}$--groups,
it is easy to show that if a knot group is an $R^*$--group then it is an $\bar{R}$--group (Lemma \ref{lem:R*Rbar}), and
we can state the following from precedent works \cite{HT,MT,T}.

\begin{theorem}\label{thm:R*Rbar}
There exist infinitely many hyperbolic knots whose knot groups are not $R^*$--groups but $\bar{R}$--groups.
\end{theorem}

%

Suppose that $G(K)$ is not an $R^*$--group, 
namely it has a generalized torsion element. 
In the course of the proof of Theorem~\ref{thm:RRbar}, 
we will prove Theorem~\ref{thm:g-torsion_order_2} below which determines knot groups with generalized torsion elements of order two. 
Before stating the result, we need a few definitions. 

A torus knot space is said to be of \textit{even type\/} if it is the exterior
of a torus knot $T(p,q)$ with $p$ or $q$ even.
It is of \textit{odd type\/} otherwise.
Similarly, a cable space of \textit{even type\/} is defined as the exterior of $T(p,q)$ with $p$ even,
which lies  on a torus $S^1\times \partial D_0$ and runs $p$ times along $S^1$ in a solid torus $S^1\times D^2$, where $D_0\subset D$ is a smaller disk.
Also, it is of \textit{odd type\/} otherwise.

\begin{theorem}
\label{thm:g-torsion_order_2}
Let $G(K)$ be the knot group of a knot $K$.
Then $G(K)$ has a generalized torsion element of order two if and only if 
$E(K)$ contains either a torus knot space of even type or a cable space of even type
as a decomposing piece of the torus decomposition. 
In particular, such a knot group is not  an $R$--group.  
\end{theorem}


\begin{corollary}\label{cor:hyp2}
The knot group of a hyperbolic knot does not admit a generalized torsion element of order two.
\end{corollary}

We remark that Corollary \ref{cor:hyp2} is a direct consequence of Theorem \ref{thm:g-torsion_order_2}, but
we also 
give an alternative proof using stable commutator length for an independent interest. 

In the forthcoming paper \cite{HMT}, we will 
classify generalized torsion elements of order two in $3$--manifold groups.

\section{$\bar{R}$--groups and $R^*$--groups}\label{sec:Rbar}

Let us recall a characterization of $\bar{R}$--groups given in \cite{FW}.
Theorem 3.2 of \cite{FW} claims that
$G$ is an $\bar{R}$--group if and only if
$G$ is an $R$--group and
for every $x$, $y$ with $y\ne 1$ of $G$ and $m,n\in \mathbb{Z}$,
the Baumslag--Solitar relation
$x^{-1}y^mx=y^n$ implies $m=n$.
Throughout  the paper, we use this description.

\begin{lemma}[\cite{FW}]
\label{Rbar_R}
If a group $G$ is an $\bar{R}$--group, then $G$ is an $R$--group.
\end{lemma}

\begin{proof}
Assume that $x^n=y^n$ for some $n\ne 0$. We may assume $n>0$, and show that $x=y$.

Since $G$ is torsion-free (see Section \ref{sec:intro}), we may assume that $x,y\ne 1$.
Let us consider the isolator subset $I\langle x\rangle$ of the cyclic group $\langle x\rangle$.
Then $y \in I\langle x\rangle$.

We claim that $y$ lies in the normalizer of $I\langle x\rangle$.
If $g\in I\langle x\rangle$, then $g^i=x^j$ for some integers $i>0$ and $j$.
Thus 
\[
(y^{-1}g y)^{in}=y^{-1}g^{in}y=y^{-1}x^{jn}y=y^{-1}y^{jn}y=y^{jn}=x^{jn}\in \langle x\rangle,
\]
so $y^{-1}gy\in I\langle x\rangle$.

By the definition of $\bar{R}$--group, the normalizer and the centralizer of $I\langle x\rangle$ coincide.
Hence $y$ lies in the centralizer of $I\langle x\rangle$.

Clearly, $x\in I\langle x\rangle$, so $x$ and $y$ commute.
This implies 
\[
(xy^{-1})^n=x^ny^{-n}=1.
\]
Again, since $G$ is torsion-free, we have $x=y$.
\end{proof}

Thus, the class of $\bar{R}$--groups is contained in that of $R$--groups.
Proposition 3.11 of \cite{FW} shows that there is a difference between these two classes.
(Certain extensions of a torsion-free abelian group of rank one by $\mathbb{Z}$ are typical examples.)

In this section, we will prove Theorem \ref{thm:R*Rbar}
after discussing the relationship between $R^*$--groups and $\bar{R}$--groups.
Throughout the paper, $[x,y]=x^{-1}y^{-1}xy$.

\begin{lemma}\label{lem:R*R}
If a group $G$ is an $R^*$--group, then $G$ is an $R$--group.
\end{lemma}

\begin{proof}
Suppose that $x^n=y^n$ for some $n\ne 0$.
Then $[x,y^n]=1$.
The commutator identity implies that
$[x,y^n]$ is a product of conjugates of $[x,y]$ (see \cite{NR}).
Hence if $[x,y]\ne 1$, then $[x,y]$ is a generalized torsion element, a contradiction.

If $[x,y]=1$, then $x^n y^{-n}=1$ implies $(xy^{-1})^n=1$.
Since $G$ is torsion-free, $x=y$.
\end{proof}

\begin{lemma}\label{lem:JS-bs}
Let $G(K)$ be a knot group.
For $x, y\in G(K)$ and $m, n \in \mathbb{Z}$,
the Baumslag--Solitar equation 
$x^{-1}y^mx=y^n$ implies that $y=1$ or $m=\pm n$.
\end{lemma}

\begin{proof}
This immediately follows from \cite[Theorem VI.2.1]{JS} or \cite{Sha}.
\end{proof}

\begin{lemma}\label{lem:R*Rbar}
Let $G(K)$ be a knot group.
If $G(K)$ is not an $\bar{R}$--group, then
either $G(K)$ is not an $R$--group, or
$G(K)$ contains a generalized torsion element of order two.
In particular, 
if $G(K)$ is an $R^*$--group, then $G(K)$ is an $\bar{R}$--group.
\end{lemma}

\begin{proof}
Assume that $G(K)$ is an $R$--group.
Since $G(K)$ is not an $\bar{R}$--group,
there exist $x$ and $y\ne 1$, and integers $m\ne n$ such that
$x^{-1}y^mx=y^n$.
By Lemma \ref{lem:JS-bs}, $m=-n$.
Hence $(x^{-1}y^mx) y^m=1$.
Since $m\ne 0$ and $G(K)$ is torsion-free, 
$y^m\ (\ne 1)$ gives a generalized torsion element of order two.

Assume that $G(K)$ is an $R^*$--group.
Then it is an $R$--group (by Lemma \ref{lem:R*R}) without generalized torsion (by definition).
Thus the first assertion shows that $G(K)$ is an $\bar{R}$--group.
\end{proof}

Among knot groups, there is a huge difference between $R^*$--groups and $\bar{R}$--groups
as claimed in Theorem \ref{thm:R*Rbar}.

\begin{proof}[Proof of Theorem \ref{thm:R*Rbar}]
By Theorems \ref{thm:main} and \ref{thm:RRbar}, we know that
the knot group of a hyperbolic knot is an $\bar{R}$--group.
On the other hand, there are plenty examples of hyperbolic knots
whose knot groups admit generalized torsion, such as
the negative twist knots and twisted torus knots (see \cite{HT, MT2,T}).
\end{proof}

\section{Generalized torsion elements}\label{sec:g-torsion}

In this paper, a generalized torsion element of order two plays a key role.
Let $g$ be a generalized torsion element in a group $G$.
If $g$ has order two, then there exist $a$ and $b$ such that $g^ag^b=1$.
By taking a conjugation with $a^{-1}$, we have $gg^{ba^{-1}}=1$, so $g^{ba^{-1}}=g^{-1}$.
In other words, $g$ is conjugate to its inverse, and conversely,
such a non-trivial element gives a generalized torsion element of order two.

We first give two examples for later use.

\begin{example}
\label{ex:g2}
\begin{itemize}
\item[(1)]
Let $E(K)$ be a torus knot space of even type.
That is,  $K$  is a torus knot  $T(p,q)$ with $p$ even.
Then the knot group $G(K)$ has a presentation $\langle a, b \mid a^p=b^q\rangle$.
Let $p=2r$.
Then $[a^p,b]=[a^{2r},b]=1$, but
\[
[a^{2r},b]=[a^r,b]^{a^r}[a^r,b].
\]
We claim $[a^r,b]\ne 1$ in $G(K)$.
Let $\phi\colon G(K)\to \langle a,b\mid a^p=b^q=1\rangle=\mathbb{Z}_p*\mathbb{Z}_q$
be the natural projection.
Then $\phi([a^r,b])=a^{r}b^{-1}a^rb$ is reduced, so nontrivial (see \cite{LS}).

Thus we have shown that $[a^r,b]$ is a generalized torsion element of order two.
\item[(2)]
Let $C(p,q)$ be a cable space of even type.
It is the exterior of $T(p,q)$ in a solid torus  $S^1\times D^2$ with $p=2r$,
where $T(p,q)\subset S^1\times \partial D_0$, $D_0\subset D^2$.
Then $G=\pi_1(C(p,q))$ has a presentation $\langle a,b,c \mid [b,c]=1, b^qc^p=a^p\rangle$.
We choose these generators so that $a$ represents the core of $S^1\times D^2$,
and $b$ and $c$ lie on $S^1\times \partial D^2$ with $b=\{*\}\times \partial D^2$, $c=S^1\times \{*\}$.

Again, $[a^p,b]=1$.
We show $[a^r,b]\ne 1$ in $G$.
Consider the natural projection
\begin{align*}
\phi\colon G&\to \langle a,b,c \mid [b,c]=1, b^qc^p=a^p=1\rangle\\
&=\langle a \mid a^p=1\rangle *
\langle b,c \mid [b,c]=1, b^qc^p=1\rangle\\
&=\mathbb{Z}_p * \mathbb{Z}.
\end{align*}
The image corresponds to the fundamental group of the annulus with one cone point of index $p$.
Then $\phi([a^r,b])= a^{-r}b^{-1}a^rb$ is reduced, which is non-trivial.
Hence as in (1), $[a^r,b]$ is a generalized torsion element of order two.
\end{itemize}
\end{example}

In the next section, we prove
that a torus knot space or a cable space contains
a generalized torsion element of order two in its fundamental group if and only if it is of even type (Lemma \ref{lem:tc}),
by evaluating the stable commutator length of an element in the commutator subgroup.

The next lemma claims that if the knot group admits a generalized torsion element of order two, 
then there exists an essential (singular) map from the Klein bottle into the knot exterior.

\begin{lemma}\label{lem:kbmap}
Let $K$ be a knot with exterior $E(K)$ and $G(K)=\pi_1(E(K))$.
If $G(K)$ admits a generalized torsion element of order two,
then there exists a singular map $f\colon F\to E(K)$, where $F$ is the Klein bottle, such that
the induced homomorphism $f_*\colon \pi_1(F)\to G(K)$ is injective.
\end{lemma}

\begin{proof}
Let $y$ be a generalized torsion element of order two in $G(K)$.
Then there exists $x$ such that  $x^{-1}yx=y^{-1}$. Since $y\ne 1$ and $G(K)$ is torsion-free, we have $x\ne 1$.
We also use the same symbols $x$ and $y$ to denote the loops with the base point $p_0$.

For the Klein bottle $F$, take two loops $a$ and $b$ meeting in a single point $q_0$ so that
$a$ is orientation-reversing but $b$ orientation-preserving.
Then they give a presentation $\pi_1(F)=\langle a,b\mid a^{-1}ba=b^{-1}\rangle$ based on the point $q_0$.
Let $f$ be a map sending $q_0$, $a$ and $b$ to  $p_0$, $x$ and $y$, respectively.
Since the image  $x^{-1}yxy$ of the loop $a^{-1}bab$ is null-homotopic in $E(K)$,
$f$ extends to a map on $F$.

We claim that
the induced homomorphism $f_*\colon \pi_1(F)\to G(K)$ is injective.
Note that any element of $\pi_1(F)$ is written as $a^ib^j$ for some integers $i$ and $j$.
Assume that $f_*$ is not injective.
Then there exists a non-trivial element $a^ib^j$ such that $f_*(a^ib^j)=x^iy^j=1$.
Since $x$ and $y$ are not torsions, $i\ne 0$ and $j\ne 0$.

On the other hand, the relation $x^{-1}yx=y^{-1}$ gives
$x^{-1}y^jx=y^{-j}$.
Since $y^j=x^{-i}$, we have $y^j=y^{-j}$, so $y^{2j}=1$.
This is impossible.
\end{proof}

For a Haken $3$--manifold  with incompressible boundary,
there exists the characteristic submanifold  $V$ by Jaco--Shalen \cite{JS} and Johannson \cite{J}.
We restrict ourselves to the exterior of a non-trivial knot.
Then $V$ is a disjoint union of Seifert fibered manifolds.
More precisely, each component is either a torus knot space, a cable space, a composing space  or $\mathrm{(torus)}\times I$.
There is a slight difference between two theories of \cite{JS} and \cite{J}.
For the knot exterior $E(K)$ of a hyperbolic knot $K$,
$V$ is empty in \cite{JS}, but  $V=\partial E(K)\times I$ in \cite{J}.


\begin{proposition}
\label{prop:g-torsion2}
If $G(K)$ admits a generalized torsion element of order two,
then
$E(K)$ contains either a torus knot space of even type or
a cable space of even type as a component of the characteristic submanifold.
\end{proposition}

\begin{proof}
By Lemma \ref{lem:kbmap}, there exists a map $f\colon F\to E(K)$, where $F$ is the Klein bottle, such
that $f_*$ is injective.
Then Corollary 13.2 of \cite{J} claims that 
$f$ is homotopic to a map $g$ with the image contained in the characteristic submanifold $V$ of $E(K)$.

Let $S$ be the component of $V$ which contains the image of $g$.
Then $\pi_1(S)$ admits a generalized torsion element of order two, because $\pi_1(F)$ contains such an element
and $g_*$ is injective.
There are only four possibilities of $S$: a torus knot space, a cable space, a composing space or $\mathrm{(torus)}\times I$.
However, a composing space and $\mathrm{(torus)}\times I$ have bi-orderable fundamental groups.
Hence there is no generalized torsion there.
Also, Lemma \ref{lem:tc} shows that
if a torus knot space or a cable space admits a generalized torsion element of order two, then
it is of even type.
Thus  $V$ contains a torus knot space of even type or a cable space of even type as a component.
\end{proof}

\begin{proof}[Proof of Theorem \ref{thm:RRbar}]
By Lemma \ref{Rbar_R}, any $\bar{R}$--group is an $R$--group.
We prove the converse for a knot group.

Suppose that $G(K)$ is an $R$--group.
By Theorem \ref{thm:main},
$E(K)$ contains neither a torus knot space nor a cable space as a component of the characteristic submanifold.

Assume that $G(K)$ is not an $\bar{R}$--group for a contradiction.
By Lemma \ref{lem:R*Rbar},  $G(K)$ admits a generalized torsion element of order two.
Then Proposition \ref{prop:g-torsion2} immediately gives a contradiction.
\end{proof}

\medskip

\begin{proof}[Proof of Theorem \ref{thm:g-torsion_order_2}]
Let us observe the ``if part''. 
Let $X$ be a decomposing piece of $E(K)$, which may be $E(K)$ itself. 
Assume that $X$ is either a torus knot space of even type or a cable space of even type. 
Then as shown in Example~\ref{ex:g2}, 
$\pi_1(X)$ has a generalized torsion element of order two. 
Since $\pi_1(X)$ is a subgroup of $G(K)$, 
$G(K)$ also has a generalized torsion element of order two. 

The ``only if'' part of Theorem~\ref{thm:g-torsion_order_2} follows from Proposition~\ref{prop:g-torsion2}. 
\end{proof}

\section{Stable commutator length}\label{sec:scl}

We quickly review the definition of stable commutator length (\cite{C}).

Let $G$ be a group and $g\in [G,G]$.
Then the \textit{stable commutator length\/} of $g$ is defined to be
\[
\mathrm{scl}_G(g)=\lim_{n\to \infty}\frac{\mathrm{cl}_G(g^n)}{n},
\]
where $\mathrm{cl}_G(a)$ denotes the commutator length of $a$, that is,
the smallest number of commutators whose product gives $a$.
For $g\not\in [G,G]$, $\mathrm{scl}_G(g)$ can be defined to be
$\mathrm{scl}_G(g^k)/k$ if $g^k\in [G,G]$, or $\infty$, otherwise.

For a knot group $G(K)$, any generalized torsion element lies in $[G(K),G(K)]$.
For, the equation (\ref{eq:g-torsion}) implies $n[g]=0\in H_1(E(K))=\mathbb{Z}$ under the abelianization, so $[g]=0$.

\begin{lemma}[\cite{IMT0}]\label{lem:imt}
Let $G$ be a group.
If $g$ is a generalized torsion element of order $k$, then 
\[
\mathrm{scl}_G(g)\le \frac{1}{2}-\frac{1}{k}.
\]
In particular, if $g$ has order two, then $\mathrm{scl}_G(g)=0$.
\end{lemma}

As mentioned in the first paragraph of Section \ref{sec:g-torsion}, an element $g$ is
a generalized torsion element of order two if and only if $g$ is conjugate to $g^{-1}$.
Hence, for any homogeneous quasimorphism $\phi \colon G \to \mathbb{R}$, $\phi(g) = \phi(g^{-1}) = -\phi(g)$, i.e. $\phi(g) = 0$.  
This also implies that $\mathrm{scl}_G(g)=0$ through Bavard's duality theorem \cite{Ba}.
Such an element often appears in the study of stable commutator length as an exceptional case.

\begin{lemma}\label{lem:tc}
A torus knot space or a cable space contains a generalized torsion element of order two in its fundamental group
if and only if it is of even type.
\end{lemma}

\begin{proof}
Example \ref{ex:g2} shows that a torus knot space or a cable space of even type
 contains a generalized torsion element of order two in its fundamental group.
 
 Conversely, consider a torus knot space of odd type.
 That is, let $K$ be a torus knot $T(p,q)$ with $p,q$ odd.
The knot group $G(K)$ has a presentation $\langle a,b \mid a^p=b^q\rangle$.
Let $g$ be a generalized torsion element of order $k\ (\ge 2)$ in $G(K)$.

\begin{claim}\label{cl:k2}
$k>2$.
\end{claim}

\begin{proof}[Proof of Claim \ref{cl:k2}]
Let $\phi\colon G(K) \to H=\langle a,b \mid a^p=b^q=1\rangle=\mathbb{Z}_p * \mathbb{Z}_q$ be the natural projection.
This map is induced by collapsing each fiber of the Seifert fibration to a point.
Equivalently, $\ker \phi$ is the center of $G(K)$, which is the infinite cyclic normal subgroup generated by a regular fiber $h\ (=a^p=b^q)$.

First, assume that $\phi(g)$ is not conjugate into one factor of $\mathbb{Z}_p*\mathbb{Z}_q$.
If $\phi(g)=a_1b_1\dots a_Lb_L$ with $a_i\in \langle a\rangle=\mathbb{Z}_p$, $b_i\in \langle b\rangle=\mathbb{Z}_q$,
$a_i\ne 1$, $b_i\ne 1$, and $L\ge 1$ (or, $\phi(g)=b_1a_1\dots b_La_L$),
then Theorem 3.1 of \cite{Ch} (or \cite[Theorem F]{CH})
claims that
\begin{equation}\label{eq:scl}
\mathrm{scl}_H(\phi(g))\ge \frac{1}{2}-\frac{1}{N},
\end{equation}
where
$N$ is the minimum order of $a_i, b_i$.
Since $p$ and $q$ are odd, $N\ge 3$.
Thus $\mathrm{scl}_H(\phi(g))>0$.
By the monotonicity of the stable commutator length (\cite[Lemma 2.4]{C}), 
we have $\mathrm{scl}_{G(K)}(g)\ge \mathrm{scl}_H(\phi(g))>0$.
Then $k>2$ by Lemma \ref{lem:imt}.

Otherwise, $\phi(g)$, after a conjugation if necessary, lies in one factor.
This implies that $\phi(g)=a^i$ or $b^i$, so
$g=a^ih^j$ or $b^ih^j$ for some integers $i, j$.
Then $g=a^{i+pj}$ or $b^{i+qj}$.
In $H_1(E(K))=\mathbb{Z}$, $[g]=(i+pj)[a]$ or $(i+qj)[b]$.
We recall that $[g]=0$, $[a]=q$ and $[b]=p$ in $H_1(E(K))$.
Thus $(i+pj)q=0$ or $(i+qj)p=0$, so $i+pj=0$ or $i+qj=0$.
However, this implies $g=1$, a contradiction.
\end{proof}

Next, let $C(p,q)$ be a cable space of odd type,
and let $G=\pi_1(C(p,q))=\langle a, b,c \mid [b,c]=1, b^qc^p=a^p\rangle$.
As above, consider the natural projection $\phi\colon G\to H=\langle a, b,c \mid [b,c]=1, b^qc^p=a^p=1\rangle
=\mathbb{Z}_p*\mathbb{Z}$ (see Example \ref{ex:g2}(2)).
Then $\ker \phi$ is the center of $G$, which is the infinite cyclic normal subgroup generated by a regular fiber.

Let $g$ be a generalized torsion element of order $k$ in $G$.
We can show that $k>2$ as in the proof of Claim \ref{cl:k2}.

If $\phi(g)$ has the cyclically reduced form of length at least two, then 
we still have the evaluation (\ref{eq:scl}), whereas $N$ is the minimum order of $a_i\in \mathbb{Z}_p$ (\cite{C,CH}).
Thus $k>2$ as above.

Suppose that $\phi(g)$ lies in one factor of $\mathbb{Z}_p*\mathbb{Z}$.
Let $d$ be a generator of the second factor $\langle b,c\mid [b,c]=1, b^qc^p=1\rangle=\mathbb{Z}$.
(Explicitly, take integers $r,s$ such that $pr -qs=1$, and then $d=b^rc^s$.)
Let $h\in G$ be the regular fiber, which is equal to $a^p\ (=b^qc^p)$.
Then $g=a^ih^j=a^{i+pj}$ or $d^ih^j=b^{ri+qj}c^{si+pj}$.

Note that $H_1(C(p,q))=\mathbb{Z}\oplus \mathbb{Z}$.
Let $\mu$ be the meridian of $T(p,q)$ in $S^1\times D^2$.
Then $[\mu]$ and $[c]$ generate $H_1(C(p,q))$, and $[b]=(p,0)$, $[c]=(0,1)$, and 
$[a]=(q,1)$ in $H_1(C(p,q))$ with suitable orientations.

If $g=a^{i+pj}$, then $[g]=((i+pj)q,i+pj)$.
Thus $i+pj=0$, so $g=1$, a contradiction.
If $g=b^{ri+qj}c^{si+pj}$, then $[g]=((ri+qj)p,si+pj)$.
Hence $ri+qj=si+pj=0$, which gives $g=1$ again.
\end{proof}

%
%
%
%
%
%
%



\begin{proof}[Proof of Corollary \ref{cor:hyp2}]
Let $g$ be a generalized torsion element of order two.
Assume that $g$ is not peripheral.
That is, it is not conjugate into the peripheral subgroup.
Then, by \cite{IMT}, there exists $\delta_g>0$ such that
\[
\mathrm{scl}_{\pi_1(K(r))}(p_r(g))\ge \delta_g\quad \textrm{whenever}\ p_r(g) \ne 1
\]
for any hyperbolic surgery $r$, where
$p_r\colon G(K)\to \pi_1(K(r))$ is the natural projection.
By the monotonicity of the stable commutator length (\cite{C}),
$\mathrm{scl}_{G(K)}(g)\ge \mathrm{scl}_{\pi_1(K(r))}(p_r(g))$,
so  we have $\mathrm{scl}_{G(K)}(g)>0$, contradicting Lemma \ref{lem:imt}.

Hence we assume that $g$ is peripheral.
Since a generalized torsion element lies in $[G(K),G(K)]$,
$g$ is conjugate to a power of the longitude $\lambda$.
We note that
\[
\mathrm{scl}_{G(K)}(\lambda)=g(K)-\frac{1}{2}\ge \frac{1}{2}
\]
by \cite[Proposition 4.4]{C}, where $g(K)$ denote the genus of $K$.
This implies that $\lambda$ is not a generalized torsion element,
neither is $g$.
\end{proof}

\begin{remark}
Corollary \ref{cor:hyp2} also follows from \cite[Lemma 8.13]{CH}.
It claims that the knot group $G(K)$ of a hyperbolic knot has a strong spectral gap relative to the peripheral subgroup,
and that the relative stable commutator length vanishes on $g$ if and only if $g$ is conjugate into the peripheral subgroup.
\end{remark}

\section*{acknowledgement}

We would like to thank Lvzhou Chen for helpful communication.

\bibliographystyle{alpha}

\end{document}